\documentclass[12pt,a4paper]{amsart}
\usepackage[utf8]{inputenc}
\usepackage[american]{babel}
\usepackage{amsmath,amssymb,amsthm}
\usepackage{enumerate}
\usepackage{geometry}
\usepackage[pdfencoding=auto]{hyperref}
\usepackage[T1, T2A]{fontenc}

\usepackage{csquotes}
\numberwithin{equation}{section}

\renewcommand{\mkbegdispquote}[2]{\itshape}

\date{\today}

\usepackage{graphicx,caption,color,xcolor}
\usepackage{epigraph}
\usepackage{amsmath,amsthm,amssymb,amsfonts}
\usepackage{enumerate}
\usepackage{mathtools}
\usepackage{hyperref}
\usepackage{cleveref}
\usepackage{todonotes}

\theoremstyle{plain}
\newtheorem{thm}{Theorem}[section]
\newtheorem{lm}[thm]{Lemma}
\newtheorem{prop}[thm]{Proposition}

\newtheorem*{theorem*}{Theorem}

\theoremstyle{definition}
\newtheorem{df}[thm]{Definition}

\newtheorem{rem}[thm]{Remark}

\newcommand{\IC}{\ensuremath{\mathbb{C}}}

\newcommand{\II}{\ensuremath{{I}}}
\newcommand{\IK}{\ensuremath{\mathbb{C}}}

\newcommand{\IN}{\ensuremath{\mathbb{N}}}

\newcommand{\IR}{\ensuremath{\mathbb{R}}}

\newcommand{\cA}{\ensuremath{\mathcal{A}}}

\newcommand{\cD}{\ensuremath{\mathcal{D}}}

\newcommand{\cF}{\ensuremath{\mathcal{F}}}

\newcommand{\cH}{\ensuremath{\mathcal{H}}}

\newcommand{\cK}{\ensuremath{\mathcal{K}}}
\newcommand{\cL}{\ensuremath{\mathcal{L}}}

\newcommand{\cN}{\ensuremath{\mathcal{N}}}

\newcommand{\cR}{\ensuremath{\mathcal{R}}}

\renewcommand{\subset}{\ensuremath{\subseteq}}
\renewcommand{\supset}{\ensuremath{\supseteq}}
\newcommand{\norm}[1]{\left\lVert#1\right\rVert}
\newcommand{\abs}[1]{\left\lvert#1\right\rvert}
\DeclarePairedDelimiterX{\inner}[2]{\langle}{\rangle}{#1, #2}

\newcommand{\ten}{\hat{\otimes}}
\begin{document}
\author[Lenz]{Daniel Lenz}
\address[Lenz]{Institute of Mathematics, Friedrich Schiller University Jena, Germany}
\email{daniel.lenz@uni-jena.de}

\author[Weinmann]{Timon Weinmann}
\address[Weinmann]{Department of Mathematics and Computer Science, St. Petersburg State University, Russia}
\email{st082214@student.spbu.ru}
\author[Wirth]{Melchior Wirth}
\address[Wirth]{Institute of Mathematics, Friedrich Schiller University Jena, Germany\\Address at the time of publication: IST Austria, Klosterneuburg, Austria}
\email{melchior.wirth@ist.ac.at}

\title[Self-Adjoint Extensions of Bipartite Hamiltonians]{Self-Adjoint Extensions of Bipartite Hamiltonians}
\begin{abstract}
We compute the deficiency spaces of operators of the form $H_A{\hat{\otimes}} I + I{\hat{\otimes}} H_B$, for symmetric $H_A$ and self-adjoint $H_B$.  This enables us to construct self-adjoint extensions (if they exist) by means of von Neumann's theory.
The structure of the deficiency spaces for this case was asserted already in \cite{IMPP14}, but only proven under the restriction of $H_B$ having discrete, non-degenerate spectrum.
\end{abstract}
\maketitle

\setcounter{tocdepth}{1}

\section*{Introduction}\label{sec:intro}

In quantum mechanics, the dynamics of a system is governed by the Schrödinger equation
\begin{equation*}
\partial_t \psi_t=-iH\psi_t,
\end{equation*}
where $H$ is a self-adjoint operator on a Hilbert space $\mathcal{H}$, called the Hamiltonian, and $\psi_t\in \mathcal{H}$ is the wave function at time $t$. Its time evolution is given by
\begin{equation*}
\psi_t=e^{-itH}\psi_0.
\end{equation*}

In many situations however, physical reasoning yields merely a symmetric, rather than self-adjoint, operator,  defined on a subspace of sufficiently regular functions. It is then natural to ask whether this operator has self-adjoint extensions, and if so, how many.

This question was completely answered by von Neumann, whose extension theory states the following: A symmetric operator $H$ has self-adjoint extensions if and only if the dimensions of the deficiency spaces $\cN(H^\ast-i)$ and $\cN(H^\ast+i)$ coincide, and in this case, the self-adjoint extensions of $H$ are parametrized by the unitary operators from $\cN(H^\ast-i)$ to $\cN(H^\ast+i)$.

Now, given two systems $A$ and $B$, we can consider the composite system $AB$. It is modeled on the tensor product of the Hilbert spaces of the individual systems $A$ and $B$. In the simplest case when there is no interaction between $A$ and $B$, the time evolution of the composite system is separable, that is,
\begin{align*}
\psi_t=(e^{-itH_A}\ten e^{-itH_B})\psi_0.
\end{align*}
The corresponding Hamiltonian is 
\begin{align*}
H_{AB}=\overline{H_A\hat\otimes \II+\II\hat\otimes H_B}.
\end{align*}
As such, operators of this form are the prototype of Hamiltonians governing the time evolution of composite systems with interactions and are therefore key in understanding the self-adjoint realizations of Hamiltonians of quantum many-particle systems and in particular open quantum systems \cite{BP02,Dav76}.

In the light of the discussion above, one is lead to study self-adjoint extensions of operators of the form $H_A\otimes\II+\II\otimes H_B$ when $H_A$ and $H_B$ are merely symmetric (in general, the time evolution generated by such a self-adjoint extension will not be separable).

In the case when one of the operators is self-adjoint, this problem was considered by Ibort, Marmo and Pérez-Pardo. They state the following result (\cite[Theorem 2.3]{IBB}, \cite[Theorem 1]{IMPP14}).

\begin{theorem*}
Let $H_A$ be a symmetric operator on $\cH_A$,  $H_B$ a self-adjoint operator on $\cH_B$ and define $H_{AB}$ on $\cH_{AB}=\cH_A\ten \cH_B$ by
$$H_{AB}=H_A\ten\II+\II\ten H_B.$$ Let $\mathcal{N}_{A\pm}=\cN(H_A^\ast\mp i)$ be the deficiency spaces of system $H_A$. The deficiency spaces $\mathcal{N}_{AB\pm}=\cN(H_{AB}^\ast\mp i)$ of $H_{AB}$ then satisfy $$\mathcal{N}_{AB\pm}\simeq\mathcal{N}_{A\pm}\ten\cH_B.$$
\end{theorem*}

However, in their proof they restrict themselves to the case when the spectrum of $H_B$ consists solely of simple eigenvalues, and only state that the general case can be treated by a judicious use of the spectral theorem. In this article, we give a complete proof of he preceding theorem for general self-adjoint $H_B$.

In the case when the operator $H_B$ is (a restriction of) the Laplacian, the assumption of discrete spectrum covers only compact geometries, while our proof justifies the result also for non-compact geometries. For example, if the system $B$ consists of a single free particle in $\IR^n$, then the Hamiltonian $H_B$ is the Laplacian, which is essentially self-adjoint on $C_c^\infty(\IR^n)$, but the spectrum of its closure is not discrete.

Let us outline the proof strategy. First, the spectral theorem allows us to view $H_B$ as a multiplication operator $M_\phi$ on $L^2(\Omega,\mu)$ for some measure $\mu$. The tensor product $\cH_A\hat\otimes L^2(\Omega,\mu)$ can be identified with the Bochner Lebesgue space $L^2(\Omega,\mu;\cH_A)$. Under this identification, the operator $H_{AB}$ acts as
\begin{align*}
(H_{AB}\psi)(\omega)=H_A\psi(\omega)+\phi(\omega)\psi(\omega).
\end{align*}
These identifications make it possible to reduce the asserted isomorphism for the deficiency spaces to a similar computation as in the proof of Ibort, Marmo and Pérez-Pardo.

There are two main difficulties to overcome. First, while the identification of tensor products of operators on $\cH_A\ten L^2(\Omega,\mu)$ with operators on $L^2(\Omega,\mu;\cH_A)$ is fairly obvious in the bounded case, we deal with unbounded operators and, as usual, more care is required to determine the correct domains. This is done in Section \ref{sec:adjoint}.

Second, the asserted isomorphism for the deficiency spaces comes from fiberwise isomorphisms. Then one has to prove that these isomorphisms can be chosen so that they depend measurably on the base point. In the discrete case this is of course obvious, but it becomes non-trivial in the general case at hand. This problem is resolved in Section \ref{sec:measurable}. Finally, the the proof is completed in Section \ref{sec:deficiency}.

The self-adjoint extensions of operators of the form $H_{AB}$ as above were also described in \cite{BBMNP18}, using a completely different approach based on boundary triplets. More precisely, given a boundary triplet for $H_A^\ast$, they construct a boundary triplet for $H_{AB}^\ast$ that respects the tensor structure.

\subsection*{Acknowledgments}
M. W. gratefully acknowledges financial support by the German Academic Scholarship Foundation (Studienstiftung des deutschen Volkes).

T.W. thanks PAO Gazprom Neft (ПАО Газпром нефть) and ORISA GmbH for their financial support in form of scholarships during his Master's and Bachelor's studies respectively.

The authors want to thank Mark Malamud for pointing out the reference \cite{BBMNP18} to them.

This work was supported by the Ministry of Science and Higher Education of the Russian Federation, agreement № 075-15-2019-1619.
\section{Preliminaries}\label{sec:prelim}
Throughout this article, all Hilbert spaces are assumed to be separable. An operator $A$ is called symmetric if it is densely defined and $A\subseteq A^\ast$.
\begin{df}[Deficiency Spaces]
Let $A$ be a symmetric operator. We call
\begin{align*}
\cN_\pm(A)=\cR(A\pm i)^\perp
\end{align*}
the \emph{deficiency spaces} of $A$ and their dimensions
\begin{align*}
d_\pm(A)=\dim \cN_\pm,
\end{align*}
the \emph{deficiency indices} of $A$.
\end{df}
In case we can rule out confusion, we shall write $\cN_\pm$ instead of $\cN_\pm(A)$.

Next we state the central theorem of von Neumann's extension theory (see, for instance, 
\cite{Wei}[Chapter 10] for an extensive treatment). In the following, $U\dot{+}V$ denotes the algebraic direct sum of subspaces $U$ and $V$, not necessarily orthogonal, while $U\oplus V$ is reserved for the orthogonal direct sum of closed subspaces.

\begin{thm}[von Neumann's Extension Theorem]
Let $A$ be a closed, symmetric operator on $\cH$. Then $A$ has self-adjoint extensions if and only if $d_+(A)=d_-(A)$. In this case, let $U\colon      \cN_+\to\cN_-$ be unitary. Define the operator $B$ by
\begin{align*}
\cD(B)&=\cD(A)\dot{+}(U+\II)\cN_+\,& B(f+Ug+g)=Af+ig-iUg.
\end{align*}
Then $A\subset B=B^*$,
and all self-adjoint extensions of $A$ arise this way.\qed
\end{thm}
This theorem yields a one-to-one correspondence between the set of self-adjoint extensions of the operator $A$ and the set of unitary operators $\cN_+\to\cN_-$, hence reduces the problem of finding self-adjoint extensions of $A$ to constructing unitary operators between the deficiency spaces. We will therefore be interested in computing the deficiency spaces of symmetric closed operators.

We also recall the spectral theorem in multiplication operator form (see for example \cite{Wei}[Chapter 8]).
\begin{thm}[Spectral Theorem]
Let $H$ be a self-adjoint operator on the Hilbert space $\cH$. There exists a $\sigma$-finite measure space $(X,\mu)$, a measurable function $\varphi\colon X\to \IR$ and a unitary operator $U\colon \cH\to L^2(X,\mu)$ such that
\begin{equation*}
H=U^*M_\varphi U.
\end{equation*}
\end{thm}

Let $\cH$ and $\cK$ be Hilbert spaces. We denote by $\cH\otimes\cK$ their algebraic tensor product and by $\cH\ten\cK$ their Hilbert space tensor product.
\begin{df}
For operators $A$ on $\cH$ and $B$ on $\cK$, we define the operator $A\otimes B$ on $\cH\ten\cK$ by
\begin{align*}
\cD(A\otimes B)&=\cD(A)\otimes\cD(B),\\
A\otimes B(\sum_j f_j\otimes g_j)&=\sum_j Af_j\otimes Bg_j.
\end{align*}
 If $A$ and $B$ are closable, then so is $A\otimes B$, and we denote its closure by $A\ten B$.
\end{df}

Let $(\Omega,\cF,\mu)$ be a measure space, $\cH$ a Hilbert space and 
$f\colon \Omega\to \cH$ measurable. We write $L^2(\Omega;\cH)$ for the Bochner-Lebesgue space of square-integrable $\cH$-valued functions.

One can construct a unitary operator $\cH\ten L^2(\Omega)\to L^2(\Omega;\cH)$ by linearly and continuously extending $e\otimes f \mapsto (\omega\mapsto f(\omega)e)$, thus  justifying the identification of $\cH\ten L^2(\Omega)$ with $L^2(\Omega;\cH)$.

\section{Finding The Adjoint}\label{sec:adjoint}
Let $H_A$ be a symmetric operator on $\cH_A$ and let $H_B$ be a self-adjoint operator on $\cH_B$. We define the operator $H_{AB}$ on $\cH_A\ten\cH_B$ by $H_{AB}=\overline{H_A\ten I + I \ten H_B}$.

The straightforward way to compute the deficiency spaces $\cN_\pm(H_{AB})=\cN(H_{AB}^*\mp i)$ of the operator $H_{AB}$ is, of course, to compute the adjoint of $H_{AB}$ and then the kernel of $H_{AB}^*\mp i$. 

First, note the following: 
\begin{align*}
\cD(H_A\ten I + I \ten H_B)&=\cD(H_A\ten I)\cap \cD(I \ten H_B)\\
&\supset \cD(H_A\otimes I)\cap \cD(I \otimes H_B)\\
&=\cD(H_A)\otimes\cD(H_B).
\end{align*}
Hence $H_A\ten I + I \ten H_B$ is densely-defined and it makes sense to consider its adjoint. Since $(H_A\ten I)^*=(H_A\otimes I)^*\supset H_A^*\otimes I$, we have $(H_A\ten I)^*\supset H_A^*\ten I$ and similarly $(I\ten H_B)^*\supset I\ten H_B$.
Combined, this means that $(H_A\ten I + I \ten H_B)^*\supset H_A^*\ten I+I\ten H_B$, and furthermore $(H_A\ten I + I \ten H_B)^*\supset \overline{H_A^*\ten I+I\ten H_B}$.
In fact, we will see that these two operators are equal. 

Before proving this, let us simplify our notation.
Without loss of generality we can assume, due to the spectral theorem, that $\cH_B=L^2(\Omega)$ and $H_B=M_\varphi$ 
for some $\sigma$-finite measure space $\Omega$ and a measurable function $\varphi\colon\Omega\to\IR$. Instead of $H_A$ and $\cH_A$, we will simply write $H$ and $\cH$ respectively.
As we have seen, we can identify $\cH\ten L^2(\Omega)$ with the Bochner-Lebesgue space $L^2(\Omega;\cH)$.
\begin{thm}\label{DAI}
Let $K$ be a closed operator on $\cH$ and let $(\Omega,\cF,\mu)$ be a measure space.
The domain of the operator $K\ten I$ on $\cH\ten L^2(\Omega)$ is given by
\begin{equation*}
\cD(K\ten \II)=\{f\in L^2(\Omega;\cH)\colon f(\omega)\in\cD(K)\text{ a.e., }\omega\mapsto K(f(\omega))\in L^2(\Omega;\cH)\}
\end{equation*}
and the operator acts as
\begin{equation*}
((K\ten \II) f)(\omega)=K(f(\omega)),
\end{equation*}
for $f\in \cD(K\ten \II)$ and a.e. $\omega\in\Omega$.
\end{thm}
\begin{proof}
Since $K$ is closed, the space $\cD(K)$ equipped with the inner product $\inner{\cdot}{\cdot}_K$ given by
\begin{equation*}
\inner{f}{g}_K=\inner{f}{g}+\inner{Kf}{Kg}
\end{equation*}
for $f,g\in\cD(A)$, is a Hilbert space.

Since $K\ten I$ is the closure of the operator $K\otimes I$, the domain $\cD(K\ten I)$ is the closure of 
$\cD(K\otimes I)=\cD(K)\otimes L^2(\Omega)$ with respect to the norm $\norm{\cdot}_{K\ten I}=\sqrt{\Vert {\cdot}\Vert^2+\Vert (K\ten I)\cdot\Vert^2}$. 

This norm, however, coincides with the norm defined on the Bochner-Lebesgue space $L^2(\Omega;(\cD(K),\inner{\cdot}{\cdot}_K))=\cD(K)\ten L^2(\Omega)$. Since $\cD(K)\ten L^2(\Omega)$
is defined as the closure of $\cD(K)\otimes L^2(\Omega)$ with respect to $\norm{\cdot}_{K\ten I}$, we have
\begin{equation*}
\cD(K\ten I)=\overline{\cD(K\otimes I)}^{K\ten I}=\overline{\cD(K)\otimes L^2(\Omega)}^{K\ten I}=\cD(K)\ten L^2(\Omega)=L^2(\Omega;\cD(K)).
\end{equation*}
Since $\int \norm{f(\omega)}^2_K d\mu(\omega)=\int \left(\norm{f(\omega)}^2+\norm{Kf(\omega)}^2 \right)d\mu(\omega)<\infty$ if and only if both
$\int \norm{f(\omega)}^2 d\mu(\omega)<\infty$ and $\int \norm{Kf(\omega)}^2 d\mu(\omega)<\infty$, we can make the following identification
\begin{equation*}
L^2(\Omega;\cD(K))=\{f\in L^2(\Omega;\cH)\colon f(\omega)\in\cD(K)\text{ a.e., } \omega\mapsto K(f(\omega))\in L^2(\Omega;\cH)\},
\end{equation*}
proving the statement about the domain. It remains to show how $K\ten I$ acts.

Let $f\in \cD(K\otimes I)=\cD(K)\otimes\cD(M_\psi)$, that is $f$, decomposes as a finite linear combination $f=\sum_k e_k\otimes f_k$. Then
\begin{align*}
(K\ten I f)(\omega)&=\sum_k ((Ke_k)\otimes f_k)(\omega)\\
&=\sum_k f_k(\omega)Ke_k\\
&=\sum_k K(f_k(\omega)e_k)\\
&=K(\sum_k f_k(\omega)e_k)\\
&=K(f(\omega)),
\end{align*}
for almost every $\omega\in\Omega$.

Now let $f\in\cD(K\ten I)$. Since $\cD(K\otimes I )$ is a core for $K\ten I$, i.e. ${\cD(K\otimes M_\psi )}$ is dense in $\cD(K\ten M_\psi )$ with respect to the graph norm,
there is a sequence $(\phi_n)_n$ in $\cD(K\otimes I )=\cD(K)\otimes\cD( I )$ that converges in $\norm{\cdot}_{K\ten I}=\sqrt{\Vert {\cdot}\Vert^2+\Vert (K\ten I)\cdot\Vert^2}$ to $f$.
In particular there is a subsequence $(\phi_{n_l})_l$ of $(\phi_n)_n$ such that $\phi_{n_l}(\omega)\to f(\omega)$ for almost every $\omega\in\Omega$, and a subsequence $(\phi_{n_{l_j}})_j$ of $(\phi_{n_l})_l$ such that almost everywhere we have
$(K\ten I\phi_{n_{l_j}})(\omega)\to(K\ten I f)(\omega)$.
 
Since $(K\ten I\phi_{n_{l_j}})(\omega)= K(\phi_{n_{l_j}}(\omega))$,
this convergence implies that for almost every $\omega$, we have
$K(\phi_{n_{l_j}}(\omega))\to(K\ten I f)(\omega)$.

Hence $K(\phi_{n_{l_j}}(\omega))\to(K\ten I f)(\omega)$ almost everywhere. Closedness of $K$ yields $f(\omega)\in\cD(K)$ and
\begin{equation*}
(K\ten I)(\omega)=\lim_j (K\ten I\phi_{n_{l_j}})(\omega)=\lim_j K(\phi_{n_{l_j}}(\omega))=K(f(\omega))
\end{equation*}
almost everywhere, which concludes the proof.
\end{proof}

\begin{prop}
Let $M_\psi$ be the operator of multiplication by the measurable function $\psi\colon \Omega\to \IC$ on $L^2(\Omega)$. The operator $I\ten M_\psi$ has domain
\begin{equation*}
\cD(I\ten M_\psi)=\{f\in L^2(\Omega;\cH)\colon \omega\mapsto\norm{f(\omega)}\in\cD(M_\psi)\}
\end{equation*}
and acts by 
\begin{equation*}
(M_\psi f)(\omega)=\psi(\omega)f(\omega),
\end{equation*}
almost everywhere.
\end{prop}
\begin{proof}
To see how the operator acts, consider the following. Let $f\in\cD(I\otimes M_\psi)=\cH\otimes\cD( M_\psi)$, then 
\begin{align*}
(I\ten M_\psi f)(\omega)&=\sum_k (e_k\otimes( M_\psi f_k))(\omega)\\
&=\sum_k \psi(\omega)f_k(\omega)e_k\\
&=\psi(\omega)\sum_k f_k(\omega)e_k\\
&=\psi(\omega)\sum_k f_k(\omega)e_k\\
&=\psi(\omega)f(\omega).
\end{align*}
Now, let $f\in\cD(I\ten M_\psi)$. Since $\cD(I\otimes M_\psi )$ is a core for $I\ten M_\psi$, i.e. ${\cD(I\otimes M_\psi )}$ is dense in $\cD(I\ten M_\psi )$ with respect to the graph norm,
there is a sequence $(\phi_n)_n$ in $\cD(I\otimes M_\psi )=\cD(I)\otimes\cD( M_\psi )$ that converges in $\norm{\cdot}_{I\ten M_\psi}=\sqrt{\Vert {\cdot}\Vert^2+\Vert (I\ten M_\psi)\cdot\Vert^2}$ to $f$.
In particular there is a subsequence $(\phi_{n_l})_l$ of $(\phi_n)_n$ such that $\phi_{n_l}(\omega)\to f(\omega)$ almost everywhere, hence
$(I\ten M_\psi\phi_{n_{l}})(\omega)=\psi(\omega)\phi_{n_{l}}(\omega)\to(I\ten M_\psi f)(\omega)$ almost everywhere.

It remains to prove our claim about the domain.

Since $I\ten M_\psi$ acts by $I\ten M_\psi f =(\omega\mapsto \psi(\omega)f(\omega))$, the inclusion ``$\subset$'' holds. 

Now, let $f\in\{g\in L^2(\Omega;\cH)\colon \omega\mapsto\norm{g(\omega)}\in\cD(M_\psi)\}$. 
Let $(\xi_n)$ be an orthonormal basis of $\cH$, then there are $\phi_n\in L^2(\Omega)$, such that
\begin{equation*}
f=\sum_n \xi_n\otimes \phi_n=\sum_n \phi_n(\cdot)\xi_n.
\end{equation*}
Since 
\begin{equation*}
\abs{\phi_{n_0}(\omega)}^2\le \sum_n \abs{\phi_n(\omega)}^2=\norm{f(\omega)}^2
\end{equation*}
for all $n_0\in \IN$, the fact that $\norm{f(\cdot)\in\cD(M_\psi)}$ implies $\phi_n\in L^2(\Omega)$ for all
$n\in\IN$.
Now, let $f_N$ be given by $f_N=\sum_{n=1}^N\phi_n(\cdot)\xi_n\in\cH\otimes\cD(M_\psi)$. Obviously, we have
\begin{equation*}
\norm{f-f_N}\xrightarrow{N\to\infty}0.
\end{equation*}
In particular, there is a subsequence $(f_{N_l})$ of $(f_N)$ such that $f_{N_l}(\omega)\xrightarrow{l\to\infty}f(\omega)$ almost everywhere.

Let $g=\phi f\in L^2(\Omega;\cH)$. Since 
\begin{equation*}
\norm{\psi(\omega)(f_{N}(\omega)-f(\omega))}=\abs{\psi(\omega)}\norm{\sum_{n=N+1}^\infty\phi_n(\omega)\xi_n}\le \abs{\psi(\omega)}\norm{f(\omega)},
\end{equation*}
Lebesgue's dominated convergence theorem yields
\begin{equation*}
\norm{I\ten M_\psi f_{N_l}-g}^2=\int\norm{\psi(\omega)(f_{N_l}(\omega)-f(\omega))}^2d\mu(\omega)\xrightarrow{l\to\infty}0.
\end{equation*}
In summary $\norm{f_{N_l}-f}\to 0$ and $\norm{I\ten M_\psi f_{N_l}-g}\to 0$. Since $I\ten M_\psi$ is closed, $f\in\cD(I\ten M_\psi)$, which proves the inclusion ``$\supset$'' and hence concludes the proof.
\end{proof}

In summary, we now know that $H^*\ten I$ and $I\ten M_\varphi$ act in the following way
\begin{align*}
((H^*\ten I) f)(\omega)&=H^*(f(\omega))\\
((I\ten M_\varphi) g)(\omega)&=\varphi(\omega)g(\omega)
\end{align*}
almost everywhere, for all $f\in \cD(H^*\ten I)$ and $g\in\cD(I\ten M_\varphi)$.
Therefore the operator $H^*\ten I +I\ten M_\varphi$ acts on $f\in\cD(H^*\ten I)\cap \cD(I\ten M_\varphi)$ by
\begin{equation*}
((H^*\ten I +I\ten M_\varphi) f)(\omega)=(\varphi(\omega)I+H^*)(f(\omega)).
\end{equation*}
almost everywhere. This extends to $\overline{H^*\ten I +I\ten M_\varphi}$, as the next lemma shows.

\begin{lm}
For every $f\in \cD(\overline{H^*\ten I +I\ten M_\varphi})$ we have
\begin{equation*}
(\overline{H^*\ten I +I\ten M_\varphi}f)(\omega)=(\varphi(\omega)I+H^*)(f(\omega))
\end{equation*}
almost everywhere.
\end{lm}
\begin{proof}
Let $\tilde H=\overline{H^*\ten I +I\ten M_\varphi}$. Take $f\in \cD(\tilde{H})$. Then there is a sequence $(f_n)$ in in the space $\cD({H^*\ten I +I\ten M_\varphi})$ such that $\norm{f_n-f}_{\tilde{H}}\to 0$. In particular, there is a subsequence $(f_{n_l})$ such that
$(\tilde{H}f_{n_l})\xrightarrow{l\to\infty}(\tilde{H}f)$ and $f_{n_{l}}\xrightarrow{l\to\infty} f$ almost everywhere. This yields 
\begin{equation*}
(\overline{H^*\ten I +I\ten M_\varphi}f)(\omega)=\varphi(\omega)f(\omega)+\lim_l H^*(f_{n_{l}}(\omega))
\end{equation*}
almost everywhere. In particular, the limit $\lim_l H^*(f_{n_{l}}(\omega))$ exists almost everywhere.
Since $H^*$ is closed, $f(\omega)\in\cD(H^*)$ and $H^*(f(\omega))=\lim_l H^*(f_{n_{l}}(\omega))$ almost everywhere. This concludes the proof.
\end{proof}

\begin{prop}\label{domain}
The domain of $\overline{H^*\ten I +I\ten M_\varphi}$ is given as follows
\begin{equation*}
\cD=\{f\in L^2(\Omega;\cH)\colon f(\omega)\in\cD(H^*)\text{ a.e., }\omega\mapsto (H^*+\varphi(\omega)I)(f(\omega))\in L^2(\Omega;\cH)\}.
\end{equation*}
\end{prop}
\begin{proof}
Let $\tilde{H}=\overline{H^*\ten I +I\ten M_\varphi}$. Again, $\cD(\tilde{H})\subset \cD$ is obvious. 

Let $f\in\cD$ and define the sequence $(E_k)$ of measurable subsets of $\Omega$ as follows

\begin{equation*}
E_k=\{\omega\in\Omega\colon \norm{H^*(f(\omega))}\le k\norm{f(\omega)}\text{ and } \abs{\varphi(\omega)}\le k\}.
\end{equation*}
Define $f_k=\chi(\cdot)_{E_k}f$. Since $\norm{H^*(f_k(\omega))}\le k\norm{f_k(\omega)}$, it is
$f_k\in\cD(H^*\ten I)$ and since $\norm{\varphi(\omega)f_k(\omega)}\le k\norm{f_k(\omega)}$ we have $f_k\in\cD(I\ten M_\varphi)$, hence $f_k\in\cD(H^*\ten I+I\ten M_\varphi)$.
Now, $\norm{f_k(\omega)-f(\omega)}=\chi(\omega)_{\Omega\setminus E_k}\norm{f(\omega)}\le \norm{f(\omega)}$
and $\norm{f_k(\omega)-f(\omega)}\xrightarrow{k\to\infty}0$ almost everywhere. Therefore, by Lebesgue's theorem
\begin{equation*}
\norm{f_k-f}^2=\int \norm{f_k(\omega)-f(\omega)}^2d\mu(\omega)\xrightarrow{k\to\infty}0.
\end{equation*}
Similarly, $\norm{(H^*+\varphi(\omega)I)(f_k(\omega)-f(\omega))}\le \norm{(H^*+\varphi(\omega)I)f(\omega)}$,
once again applying Lebesgue's theorem yields
\begin{equation*}
\norm{\tilde{H}f_k-g}^2=\int \norm{(H^*+\varphi(\omega)I)(f_k(\omega)-f(\omega))}^2d\mu(\omega)\xrightarrow{k\to\infty}0,
\end{equation*}
where $g=(\omega\mapsto (H^*+\varphi(\omega)I)f(\omega))$. Since $\tilde{H}$ is closed, we have $f\in\cD(\tilde{H})$.
\end{proof}
Before we can prove that $(H\ten I +I\ten M_\varphi)^*\subset\overline{H^*\ten I +I\ten M_\varphi}$, we need a general fact about adjoints (see \cite{Sch}[Prop. 7.26]).
\begin{lm}
Let $A$ and $B$ be densely defined, closable operators on the Hilbert spaces $\cK$ and $\cL$ respectively.
The operator $A\ten B$ on $\cK\ten \cL$ satisfies the following identity for its adjoint
\begin{equation*}
(A\ten B)^*=A^*\ten B^*.
\end{equation*}
\end{lm}

\begin{thm}
The following identity holds
\begin{equation*}
(H\ten I + I\ten M_\varphi)^*=\overline{H^*\ten I +I\ten M_\varphi}.
\end{equation*}
\end{thm}
\begin{proof}
Let $\tilde H=\overline{H^*\ten I +I\ten M_\varphi}$. Only the inclusion $\cD((H\ten I + I\ten M_\varphi)^*)\subset\cD(\tilde{H})$ is left to prove.

Let $f\in\cD((H\ten I + I\ten M_\varphi)^*)$, that is there is an ${f^*}\in L^2(\Omega;\cH)$ such that for all $g\in \cD(H\ten I + I\ten M_\varphi)$ the following holds
\begin{equation*}
\inner{f}{(H\ten I + I\ten M_\varphi)g}=\inner{{f^*}}{g}.
\end{equation*}
Let $E_k=\{\omega\in\Omega\colon \abs{\varphi(\omega)}\le k\}$ for $k\in \IN$.
Since
\begin{align*}
\cD(H\ten I)&=\{h\in L^2(\Omega;\cH)\colon h(\omega)\in\cD(H)\text{ a.e., }\omega\mapsto H(h(\omega))\in L^2(\Omega;\cH)\}\\
\cD(I\ten M_\varphi)&=\{h\in L^2(\Omega;\cH)\colon \omega\mapsto \varphi(\omega)h(\omega)\in L^2(\Omega;\cH)\}\\
\cD(H\ten I + I\ten M_\varphi)&=\cD(H\ten I)\cap\cD(I\ten M_\varphi)
\end{align*}

for $g\in \cD(H\ten I + I\ten M_\varphi)$ also $g_k=\chi_{E_k}g\in \cD(H\ten I + I\ten M_\varphi)$, hence
\begin{align*}
& \inner{f\upharpoonright_{E_k}}{(H\ten I_{L^2(E_k)} + I\ten M_{\varphi\upharpoonright_{E_k}})g\upharpoonright_{E_k}}_{L^2(E_k,\cH)}\\
=&\int_{E_k}\inner{f(\omega)}{((H\ten I + I\ten M_\varphi)g)(\omega)}d\mu(\omega)\\
=&\int_{\Omega}\inner{f(\omega)}{((H\ten I + I\ten M_\varphi)g_k)(\omega)}d\mu(\omega)\\
=&\inner{f}{(H\ten I + I\ten M_\varphi)g_k}_{L^2(\Omega;\cH)}\\
=&\inner{f^*}{g_k}_{L^2(\Omega;\cH)}\\
=&\int_{E_k}\inner{f^*(\omega)}{g(\omega)}d\mu(\omega)\\
=&\inner{f^*\upharpoonright_{E_k}}{g\upharpoonright_{E_k}}_{L^2(E_k,\cH)}.
\end{align*}
Obviously all functions in $\cD(H\ten I_{L^2(E_k)} + I\ten M_{\varphi\upharpoonright_{E_k}})$ can be extended by $0$ to functions in $\cD(H\ten I_{L^2(\Omega)} + I\ten M_{\varphi})$, and are therefore representable by restrictions of elements of $\cD(H\ten I_{L^2(\Omega)} + I\ten M_{\varphi})$. Consequently the above computation yields
\begin{equation*}
\inner{f\upharpoonright_{E_k}}{(H\ten I + I\ten M_{\varphi\upharpoonright_{E_k}})h}_{L^2(E_k,\cH)}=\inner{f^*\upharpoonright_{E_k}}{h}_{L^2(E_k,\cH)},
\end{equation*}
for all $h\in \cD(H\ten I_{L^2(E_k)} + I\ten M_{\varphi\upharpoonright_{E_k}})$, hence $f\upharpoonright_{E_k}\in\cD((H\ten I_{L^2(E_k)} + I\ten M_{\varphi\upharpoonright_{E_k}})^*)$.
By the very definition of $E_k$, the map $\varphi\upharpoonright_{E_k}$ is bounded, hence $M_{\varphi\upharpoonright_{E_k}}$ is bounded, hence $I\ten M_{\varphi\upharpoonright_{E_k}}$ is bounded by the closed graph theorem, since 
\begin{equation*}
\cD(I\ten M_{\varphi\upharpoonright_{E_k}})=\{h\in L^2(E_k,\cH)\colon \omega\mapsto \varphi(\omega)h(\omega)\in L^2(E_k,\cH)\}.
\end{equation*}
Therefore
\begin{align*}
&\cD((H\ten I_{L^2(E_k)} + I\ten M_{\varphi\upharpoonright_{E_k}})^*)\\
=&\cD((H\ten I_{L^2(E_k)})^*)\\
=&\cD(H^*\ten I_{L^2(E_k)})\\
=&\{h\in L^2(E_k,\cH)\colon h(\omega)\in\cD(H^*)\text{ a.e., }\omega\mapsto H^*(h(\omega))\in L^2(E_k,\cH)\}
\end{align*}
Now let $h\in\cD(H\ten I_{L^2(E_k,)} + I\ten M_{\varphi\upharpoonright_{E_k}})$, then
\begin{align*}
&\inner{f^*\upharpoonright_{E_k}}{h}_{L^2(E_k,\cH)}\\
=& \inner{f\upharpoonright_{E_k}}{(H\ten I_{L^2(E_k)} + I\ten M_{\varphi\upharpoonright_{E_k}})h}_{L^2(E_k,\cH)}\\
=&\int_{E_k}\inner{f\upharpoonright_{E_k}(\omega)}{((H\ten I_{L^2(E_k)} + I\ten M_{\varphi\upharpoonright_{E_k}})h)(\omega)}d\mu(\omega)\\
=&\int_{E_k}\left(\inner{H^*(f\upharpoonright_{E_k}(\omega))}{h(\omega)}+\inner{f\upharpoonright_{E_k}(\omega)}{\varphi(\omega)h(\omega)}\right)d\mu(\omega)\\
=&\inner{H^*\ten I_{L^2(E_k)}f\upharpoonright_{E_k}}{h}_{L^2(E_k,\cH)}+\int_{E_k}\inner{\varphi(\omega)(f\upharpoonright_{E_k}(\omega))}{h(\omega)}d\mu(\omega)\\
=&\inner{H^*\ten I_{L^2(E_k)}f\upharpoonright_{E_k}}{h}_{L^2(E_k,\cH)}+\int_{E_k}\inner{\varphi(\omega)(f\upharpoonright_{E_k}(\omega))}{h(\omega)}d\mu(\omega)\\
=&\inner{(H^*\ten I_{L^2(E_k)}+I\ten M_{\varphi\upharpoonright_{E_k}})f\upharpoonright_{E_k}}{h}_{L^2(E_k,\cH)}.
\end{align*}
Hence
\begin{equation*}
f^*\upharpoonright_{E_k}=(H^*\ten I_{L^2(E_k)}+I\ten M_{\varphi\upharpoonright_{E_k}})f\upharpoonright_{E_k}.
\end{equation*}
In summary we have $f(\omega)\in\cD(H^*)$ and $f^*(\omega)=H^*(f(\omega))+\varphi(\omega)$ for almost every $\omega\in E_k$.
Since $\Omega$ is covered by the $E_k$, $k\in\IN$, we have
\begin{align*}
f&\in\{k\in L^2(\Omega;\cH)\colon k(\omega)\in\cD(H^*)\text{ a.e., }\omega\mapsto H^*(k(\omega))+\varphi(\omega)k(\omega)\in L^2(\Omega;\cH)\}\}.
\end{align*}
However, by \Cref{domain}, this means $f\in\cD(\tilde{H})$.
\end{proof}

\section{Constructing A Measurable Family of Orthonormal Bases}\label{sec:measurable}
\begin{df}[Domain of regularity]
Let $A$ be an operator on the Hilbert space $\cH$. The set
\begin{equation*}
    \chi(A)=\{z\in\IK\colon      \exists c_z>0\, \forall f\in\cD(A)\colon       \norm{(A-z\II)f}\ge c_z \norm{f}\}
\end{equation*}
is called \emph{domain of regularity} of the operator $A$.
\end{df}
Note that for all closed operators $A$, the number
\begin{equation*}
d_z(A)=\dim \cR(A-z)^\perp
\end{equation*}
is constant on each connected component of the domain of regularity $\chi(A)$ of $A$ (see \cite{Wei}[Chapter 10]).
That is, for each pair of complex numbers $z,w\in\chi(A)$, there is an isometric isomorphism between $\cR(A-z)^\perp$ and $\cR(A-w)^\perp$.
Our goal is to construct such an isomorphism more or less explicitly. We will restrict ourselves here to 
closed symmetric operators.

Let $A$ be a closed symmetric operator on the separable Hilbert space $\cH$. Since $\cR(A+i)$ is a closed subspace of $\cH$, there exists $N\in\IN\cup\{\infty\}$ and $\xi_n\in \cD(A)$, $n<N$, such that $((A+i)\xi_n)_n$ is an orthonormal basis of $\cR(A+i)$.

By symmetry of $A$ we have $i\in\chi(A)$, hence $(A+i)^{-1}\colon \cR(A+i)\to \cD(A)$ is bounded.
Therefore the span of  the $\xi_n=(A+i)^{-1}(A+i)\xi_n$, $n<N$, is dense in $\cD(A)$. This gives rise to a total set in $\cR(A+z)$ for $z\in\IC^+$, as we will show next.
\begin{prop}
For every $f\in \cD(A)$ there are $\lambda_n\in\IC$, $n<N$, such that
\begin{equation*}
(A+z)f=\sum_{n<N} \lambda_n (A+z)\xi_n.
\end{equation*}
\end{prop}
\begin{proof}
Since $((A+i)\xi_n)_{n<N}$ is an orthonormal basis of $\cR(A+i)$, there are $\lambda_n\in\IC$, $n<N$, such that $\sum_n\abs{\lambda_n}^2<\infty$ and
\begin{equation*}
(A+i)f=\sum_{n<N} \lambda_n (A+i)\xi_n.
\end{equation*}

Thus, for all $k<N$,
\begin{align*}
\sum_{n=1}^k \lambda_n(A+z)\xi_n&=\sum_{n=1}^k \lambda_n (A+i)\xi_n+(z-i)\sum_{n=1}^k \lambda_n \xi_n\\
&=\sum_{n=1}^k \lambda_n (A+i)\xi_n+(z-i)(A+i)^{-1}\sum_{n=1}^k \lambda_n(A+i)\xi_n.
\end{align*}

Taking $k=N-1$ if $N<\infty$ or passing to the limit $k\to\infty$ if $N=\infty$, we obtain
\begin{equation*}
\sum_{n<N}\lambda_n (A+z)\xi_n=(A+i)f+(z-i)(A+i)^{-1}(A+i)f=(A+z)f.\qedhere
\end{equation*}
\end{proof}

Note that since $(A+z)(A+i)^{-1}$ is injective, the vectors $(A+z)\xi_1,\dots,(A+z)\xi_m$ are linearly independent for every $m<N$. We can therefore apply Gram-Schmidt orthonormalization to $((A+z)\xi_n)_{n<N}$ to obtain an orthonormal basis $(\eta_m(z))_{m<N}$ of $\cR(A+z)$. Note that all the operations in the Gram-Schmidt algorithm are continuous, in particular measurable in $z$. Thus the map $z\mapsto \eta_m(z)$ is measurable.

Denote the projection onto $\cR(A+z)$ by $P_z$, that is,
\begin{equation*}
P_z=\sum_m \inner{\eta_m(z)}{\cdot}\eta_m(z).
\end{equation*}
Now $I-P_z$ is the projection onto $\cR(A+z)^\perp=\cN(A^*+\overline{z})$ and $((I-P_z)\zeta_n)$ is total in $\cN(A^*+\overline{z})$ for every orthonormal basis of $(\zeta_n)$ of $\cH$.
Fix an orthonormal basis $(\zeta_n)$ of $\cH$ and set $\rho_n(z)=(I-P_z)\zeta_n$ for $n\in\IN$.
We now introduce a modified Gram-Schmidt algorithm which does not require the input vectors to be linearly independent.

Consider the map $\kappa\colon \IR\to \IR$ given by
\begin{equation*}
\kappa(x)=\begin{cases}
1,\text{ if }x=0\\
{x},\text{ else}
\end{cases}
\end{equation*}
and note that it is obviously measurable.
Define $(\sigma_n)$ inductively by
\begin{align*}
\sigma_1(z)&=\frac{\rho_1(z)}{\kappa(\norm{\rho_1(z)})}\\
\sigma_{n+1}(z)&=\frac{\rho_{n+1}(z)-\sum_{l=1}^n\inner{\sigma_l(z)}{\rho_{n+1}(z)}\sigma_l(z)}{\kappa(\Vert\rho_{n+1}(z)-\sum_{l=1}^n\inner{\sigma_l(z)}{\rho_{n+1}(z)}\sigma_l(z)\Vert)}.
\end{align*}
By the original Gram-Schmidt algorithm it is easy to see that those of the $\sigma_n(z)$ that do not vanish form an orthonormal basis of $\cN(A^*+\overline{z})$.
Furthermore it is evident that $z\to \sigma_n(z)$ is measurable for every $n\in\IN$.
We now want to prove that we can ``extract'' those not-vanishing $\sigma_n(z)$ in a measurable manner with respect to $z$. Remember that $d_+=\dim\cN(A^*-i)=\dim\cN(A^*+\overline{z})$.
\begin{df}
Let $n_j\colon \IC^+\to\IN$ for $j<d_+$ be defined inductively by
\begin{align*}
n_1(z)&=0,\\
n_{j+1}(z)&=\inf\{n\in\IN\colon  \sigma_n(z)\neq 0, n>n_j(z)\}.
\end{align*}
\end{df}
\begin{lm}
For every $j<d_++1$, the map $z\mapsto n_j(z)$ is measurable.
\end{lm}
\begin{proof}
We proceed inductively:  $n_1$ is constant, hence measurable.
In order to prove that $n_{j+1}$ is measurable, it suffices to show that the set
$\{z\in\IC^+\colon n_{j+1}=k\}$ is measurable for every $k\in\IN$. 
Note the following
\begin{align*}
&\{z\in\IC^+\colon n_{j+1}(z)=k\}\\
=&\{z\in\IC^+\colon n_{j}(z)<k\}\cap \bigcap_{n_j(z)<l<k}\{z\in\IC^+\colon \sigma_j(z)=0\}\cap \{z\in\IC^+\colon \sigma_k(z)=0\}\\
=&\bigcup_{m=1}^{k-1}(\{z\in\IC^+\colon n_{j}(z)=m\}\cap \bigcap_{m<l<k}\{z\in\IC^+\colon \sigma_j(z)=0\})\cap \{z\in\IC^+\colon \sigma_k(z)=0\}.
\end{align*}
By induction hypothesis, $\{z\in\IC^+\colon n_{j}(z)=m\}$ is measurable 
and therefore the set
$\{z\in\IC^+\colon n_{j+1}(z)=k\}$ is measurable as well.
\end{proof}
\begin{lm}
The map $z\mapsto \sigma_{n_j(z)}(z)$ is measurable.
\end{lm}
\begin{proof}
Let $\cA\subset \cH$ be measurable.
\begin{align*}
\{z\in\IC^+\colon \rho_{n_j}(z)\in\cA\}&=\bigcup_{k=0}^\infty\left(\{z\in\IC^+\colon n_j(z)=k\}\cap\{z\in\IC^+\colon \sigma_k(z)\in\cA\}\right).
\end{align*}
By the above lemma, this set is measurable.
\end{proof}
Note that $(\sigma_{n_j(z)}(z))_{j< d_++1}$ is an orthonormal basis of $\cN(A^*+\overline{z})$.

Let us fix an orthonormal basis $(\tau_j)_{j< d_++1}$ of $\cN(A^*-i)$. Define the unitary operator
$U_z\colon \cN(A^*+\overline{z})\to \cN(A^*-i)$ by linearly and continuously extending the operator $U_z\sigma_{n_j(z)}(z)=\tau_j$. Note that
\begin{equation*}
U_z f = \sum_j \inner{\sigma_{n_j(z)}(z)}{f}\tau_j.
\end{equation*}

for all $f\in \cN(A^*+\overline{z})$. Since the inner product is continuous, for a measurable curve $\IC^+\to\cH$,  $z\mapsto f(z)\in \cN(A^*+\overline{z})$, the map
$z\mapsto U_z f(z)$ is measurable.
In summary:
\begin{thm}\label{mu}
There is a family of unitary operators $(U_z)_{z\in\IC^+}$,
\begin{equation*}
U_z\colon \cN(A^*+\overline{z})\to\cN(A^*-i),
\end{equation*}
such that for 
all measurable $f\colon \IC^+\to\cH$, 
satisfying $f(z)\in\cN(A^*+\overline{z})$, the map $z\mapsto U_z f(z)$ is measurable.
\end{thm}

\section{Computing The Deficiency Spaces}\label{sec:deficiency}
We finally dealt with sufficiently many technicalities so that we can prove our main result.
\begin{thm}
Let $H_A$ and $H_B$ be a symmetric and a self-adjoint operator on the Hilbert space $\cH_A$ and $\cH_B$ respectively. If
$H_{AB}=\overline{H_A\ten I+I\ten H_B}$, then
\begin{equation*}
\cN_\pm(H_{AB})\simeq\cN_\pm(H_A)\ten \cH_B.
\end{equation*}
\end{thm}
\begin{proof}
Without loss of generality let $\cH_B=L^2(\Omega)$ for a $\sigma$-finite measure space $(\Omega,\cF,\mu)$ and let $H_B=M_\varphi$ be the operator of multiplication
by a real valued measurable function $\varphi\colon \Omega\to\IR$. For simplicity, denote $\cH=\cH_A$ and $H=H_A$ and $\tilde{H}=\overline{H^*\ten I +I\ten M_\varphi}$.

We will prove the existence of an isomorphism for the $\cN_+$ only; for the $\cN_-$ the proof works analogously.

If $f\in \cN_+(H_{AB})=\cN(H^*_{AB}-i)=\cN(\tilde{H}-i)$, then 
\begin{equation*}
(\tilde{H}f)(\omega)=if(\omega)\text{ a.e.,}
\end{equation*}
that is, 
\begin{align*}
H^*(f(\omega))+\varphi(\omega)f(\omega)&=if(\omega)\text{ a.e..}
\end{align*}
Hence $f(\omega)\in \cN(H^*-(i-\varphi(\omega)))$ almost everywhere. Since $\varphi$ is real-valued $i-\varphi(\omega)\in \IC^+$, thus by \Cref{mu} there is a family of unitaries 
\begin{equation*}
V_\omega\colon \cN(H^*-(i-\varphi(\omega)))\to \cN(H^*-i),
\end{equation*}

such that the map $\omega\mapsto V_\omega f(\omega)$ is measurable.
Since the $V_\omega$ are unitary, we have
\begin{equation*}
\int \norm{V_\omega f(\omega)}^2d\mu(\omega)=\int \norm{f(\omega)}^2d\mu(\omega)<\infty.
\end{equation*}
In summary, $\omega\mapsto V_\omega f(\omega)$ is in $L^2(\Omega;\cN(H^*-i))$. On the other hand, because the $V_\omega$ are onto, every $g\in L^2(\Omega;\cN(H^*-i))$ admits a representation $g(\omega)=V_\omega f(\omega)$ for some $f\in L^2(\Omega;\cH)$ satisfying $f(\omega)\in\cN(H^*-(i-\varphi(\omega)))$ almost everywhere, therefore
\begin{equation*}
\cN(\tilde{H}-i)\simeq L^2(\Omega;\cN(H^*-i)).
\end{equation*}
By closedness of $H^*$, the space $\cN(H^*-i)$, equipped with the inner product inherited from $\cH$, is a Hilbert space. In particular
\begin{equation*}
\cN(H^*-i)\ten L^2(\Omega)=L^2(\Omega;\cN(H^*-i)).\qedhere
\end{equation*}

\end{proof}

\begin{rem}
\begin{enumerate}[(i)]
\item Having constructed the deficiency spaces of $H_{AB}$, we can answer the question of existence of self-adjoint extensions of $H_{AB}$ and in case of existence, construct them by the means we examined in Section \ref{sec:prelim}.

\item Note that, because of the fact that there is a one-to-one correspondence between self-adjoint extensions of a symmetric operator and unitary extensions of its Cayley transform, there are ``more'' self-adjoint extensions of $H_{AB}$
than there are of $H_A$.

\item It still remains open to determine the deficiency spaces of $H_{AB}=H_A\otimes I+I\otimes H_B$ in the case when both $H_A$ and $H_B$ are only assumed to be symmetric. As stated in \cite{IBB}, it is quite natural to conjecture that
\begin{equation*}
\cN_{AB\pm}\simeq \cN_{A\pm}\hat\otimes \cH_B\oplus \cH_A\hat\otimes \cN_{B\pm}.
\end{equation*}
\end{enumerate}
\end{rem}

\bibliography{title}
\bibliographystyle{alpha}
\end{document}